\documentclass[english,10pt]{article}
\include{macros}

\begin{document}

\title{Some compact notations for concentration inequalities and user-friendly results}

\author
{
	Kaizheng Wang \thanks{Department of Operations Research and Financial Engineering, Princeton University, Princeton, NJ 08544, USA; Email:
		\texttt{kaizheng@princeton.edu}.}
}

\date{December 2019}

\maketitle

\begin{abstract}
This paper presents compact notations for concentration inequalities and convenient results to streamline probabilistic analysis. The new expressions describe the typical sizes and tails of random variables, allowing for simple operations without heavy use of inessential constants. They bridge classical asymptotic notations and modern non-asymptotic tail bounds together. Examples of different kinds demonstrate their efficacy.

\end{abstract}

\section{Introduction}

Concentration inequalities \citep{BLM13, Ver12, Tro12} have become the bread and butter of researchers on theoretical computer science, statistics, information theory, machine learning, signal processing and related fields. They provide non-asymptotic tail bounds for random quantities that facilitate finite-sample analysis of high dimensional problems. Unfortunately, they often contain plenty of constants that make their statements nasty. 
Some constants have to be large in order to make the results hold for finite samples. Some constants are caused by translations between different but essentially equivalent definitions of sub-Gaussianity or similar properties. Things become more daunting when we operate with more than one random variables and random events that are controlled by different concentration inequalities. The exact value of the aforementioned constants are not essential at all; what really matters are the typical sizes and tail decay of the random variables under investigation. While classical asymptotic notations such as $O_{\PP}$ and $o_{\PP}$ \citep{Van00} are able to describe typical sizes with few unwanted constants, they say nothing about tail behaviors and are thus not capable of handling a huge collection of random variables simultaneously. To make matters worse, some variants that have appeared in the literature may cause confusions, see the discussions in \cite{Jan11}.

We propose some compact notations for concentration inequalities to resolve this dilemma and develop handy tools to make probabilistic analysis quick and clean. Borrowing strength from both non-asymptotic and asymptotic characterizations, our key notation $O_{\PP} (\cdot;~ \cdot)$ has two arguments that correspond to sizes and tails of random variables. It easily converts to non-asymptotic tail bounds, admits simple operations and helps avoid repeated definitions of unspecified constants during proofs. Uniform control over a collection of random variables is also discussed. Examples throughout the paper illustrate the efficacy of new notations and results.

The rest of the paper is organized as follows. Section \ref{section-basic-notations} defines basic notations and proves simple rules for elementary operations. Section \ref{section-basic-examples} expresses common results in the new language. Section \ref{section-uniform-notations} presents advanced notations and results for uniform control. Section \ref{section-example-uniform} concludes the paper with an example of uniform convergence in statistical learning.

\subsection*{Notations}
For any real numbers $a,b \in \RR$, let $a \vee b = \max\{a,b\}$ and $a \wedge b = \min\{a,b\}$. The notation $\log$ refers to the natural logarithm. For $n\in \ZZ_+$, $\RR^n$ denotes the $n$-dimensional Euclidean space; $\SSS^{n-1}$ is the unit sphere therein; and $[n] = \{1,2,\cdots,n\}$. $C^k(\RR^n)$ is the family of real-valued functions on $\RR^n$ whose derivatives up to the $k$-th order are all continuous. For any $\bx = (x_1, \cdots, x_n)^{\top} \in \RR ^n$ and $p \geq 1$, define $\| \bx \|_p = (\sum_{i=1}^{n} |x_i|^p)^{1/p}$. For any matrix $\bA \in \RR^{n \times m}$, the matrix spectral norm is $\| \bA \|_2 = \max_{\| \bx\|_2=1} \| \bA \bx \|_2$.
For nonnegative $\{ a_n \}_{n=1}^{\infty}$ and $\{ b_n \}_{n=1}^{\infty}$, $a_n \lesssim b_n$ means $a_n \leq c b_n$, $\forall n$ for some constant $c > 0$. For any random variable $X$, define $\| X \|_{\psi_{\alpha}} = \sup_{p \geq 1} p^{-1/\alpha} \EE^{1/p} |X|^{p}$ for $\alpha \geq 1$. For any random vector $\bX \in \RR^n$, let $\| \bX \|_{\psi_2} = \sup_{\bu \in \SSS^{n - 1} } \| \bu^{\top} \bX \|_{\psi_2}$.


\section{Basic notations and operations}\label{section-basic-notations}

\begin{definition}[The $O_{\PP}$ notation]\label{definition-O}
	Let $\{ X_n \}_{n=1}^{\infty}$, $\{ Y_n \}_{n=1}^{\infty}$ be two sequences of random variables and $\{ r_n \}_{n=1}^{\infty} \subseteq (0,+\infty)$ be deterministic. We write
	\begin{align*}
	X_n = O_{\PP} (Y_n;~ r_n)
	\end{align*}
	if there exist universal constants $(C_1,C_2,N) \in (0,+\infty)^3$ and a non-decreasing function $f:~[C_2,+\infty) \to (0,+\infty)$ satisfying $\lim_{x \to +\infty} f(x) = +\infty$, such that
	\begin{align}
	\PP ( |X_n| \geq t |Y_n| ) \leq C_1 e^{- r_n f(t) }, \qquad \forall ~n \geq N,~ t \geq C_2. 
	\label{ineq-definition-O}
	\end{align}
\end{definition}


The $O_{\PP}(\cdot ;~ \cdot )$ notation is an abstraction of concentration inequalities, which usually have the form in (\ref{ineq-definition-O}). The two arguments in $O_{\PP}(\cdot ;~ \cdot )$ are called the \textit{size} and \textit{tail} arguments. It is worth pointing out that the size argument can be random. This allows for stochastic dominance and facilitates probabilistic analysis. 
Motivated by (\ref{ineq-definition-O}), we may also consider more refined characterizations of the tail behavior and define notations like $X_n = O_{\PP} (Y_n ;~ r_n, f)$. However, as we will see from the examples later, the existing one already suffices in many applications. Below we show some equivalent definitions.

\begin{fact}\label{fact-1}
	The followings are equivalent:
	\begin{enumerate}
		\item $X_n = O_{\PP} (Y_n;~r_n )$;
		\item There exist constants $C_1>0$ and $N>0$ such that
		\begin{align*}
		\forall C > 0,~~ \exists C' >0, \text{ s.t. } \PP ( |X_n| \geq C' |Y_n| ) \leq C_1 e^{- C r_n }, \qquad \forall ~n \geq N.
		\end{align*}
	\end{enumerate}
In addition, if $\PP (Y_n = 0) = 0$ for all $n$, then the conditions above are equivalent to the following: there exists a constant $C_1>0$ such that
	\begin{align*}
	\forall C > 0,~~ \exists C' >0, \text{ s.t. } \PP ( |X_n| \geq C' |Y_n| ) \leq C_1 e^{- C r_n }, \qquad \forall ~n \geq 1.
	\end{align*}
\end{fact}
\begin{proof}
The equivalence between conditions 1 and 2 is obvious. Suppose that $\PP (Y_n = 0) = 0$ for all $n$, and condition 2 holds with constants $C_1>0$ and $N>0$. For any $C>0$ we can find $C_0' > 0$ such that
\begin{align*}
 \PP ( |X_n| \geq C_0' |Y_n| ) \leq C_1 e^{- C r_n }, \qquad \forall ~n \geq N.
 \end{align*}
When $1 \leq n < N$, there exists $C_n'>0$ such that $\PP ( |X_n| / |Y_n| \geq C_n' ) \leq C_1 e^{-C r_n}$. Then $C' = \max_{ 0 \leq n < N } C_n'$ satisfies
\begin{align*}
\PP ( |X_n| \geq C' |Y_n| ) \leq C_1 e^{- C r_n }, \qquad \forall ~n \geq 1.
\end{align*}
This finishes the proof.
\end{proof}

We also introduce a weaker notion of tail bounds.

\begin{definition}[The $\hat{O}_{\PP}$ notation]\label{definition-O-hat}
	Let $\{ X_n \}_{n=1}^{\infty}$, $\{ Y_n \}_{n=1}^{\infty}$ be two sequences of random variables and $\{ r_n \}_{n=1}^{\infty} \subseteq (0,+\infty)$ be deterministic. We write
	\begin{align*}
	X_n = \hat{O}_{\PP} (Y_n;~ r_n)
	\end{align*}
	if there exist universal constants $(C_1,C_2,N) \in (0,+\infty)^3$, a non-decreasing function $f:~[C_2,+\infty) \to (0,+\infty)$ satisfying $\lim_{x \to +\infty} f(x) = +\infty$, and a positive deterministic sequence $\{ R_n \}_{n=1}^{\infty}$ tending to infinity such that
	\begin{align}
	\PP ( |X_n| \geq t |Y_n| ) \leq C_1 e^{- r_n f(t) }, \qquad \forall ~n \geq N,~ C_2 \leq  t \leq R_n. 
	\label{ineq-definition-O-hat}
	\end{align}
\end{definition}

\begin{fact}\label{fact-2}
The followings are equivalent:
\begin{enumerate}
\item $X_n = \hat{O}_{\PP} (Y_n ;~ r_n)$;
\item There exists a constant $C_1>0$ such that
\begin{align*}
\forall C > 0,~~ \exists C' >0 \text{ and } N > 0, \text{ s.t. } \PP ( |X_n| \geq C' |Y_n| ) \leq C_1 e^{- C r_n }, \qquad \forall ~n \geq N.
\end{align*}
\end{enumerate}	
\end{fact}
\begin{proof}
The proof is straightforward and thus omitted.
\end{proof}

\begin{fact}\label{fact-3}
$X_n = O_{\PP} (Y_n;~r_n )$ always implies $X_n = \hat{O}_{\PP} (Y_n ;~ r_n)$. When $\PP (Y_n = 0) = 0$ for all $n$, $X_n = O_{\PP} (Y_n;~r_n )$ and $X_n = \hat{O}_{\PP} (Y_n ;~ r_n)$ are equivalent.
\end{fact}
\begin{proof}
It is obviously true that $X_n = O_{\PP} (Y_n;~r_n )$ always implies $X_n = \hat{O}_{\PP} (Y_n ;~ r_n)$. Now we prove the other direction under the additional assumption $\PP (Y_n = 0) = 0$ for all $n$. When $X_n = \hat{O}_{\PP} (Y_n ;~ r_n)$, Fact \ref{fact-2} asserts the existence of a constant $C_1>0$ such that
\begin{align*}
\forall C > 0,~~ \exists C_0' >0 \text{ and } N > 0, \text{ s.t. } \PP ( |X_n| \geq C_0' |Y_n| ) \leq C_1 e^{- C r_n }, \qquad \forall ~n \geq N.
\end{align*}
For any $n < N$, there exists $C_n'>0$ such that $\PP ( |X_n| / |Y_n| \geq C_n' ) \leq C_1 e^{-C r_n}$. Then $C' = \max_{ 0 \leq n < N } C_n'$ satisfies
	\begin{align*}
	\PP ( |X_n| \geq C' |Y_n| ) \leq C_1 e^{- C r_n }, \qquad \forall ~n \geq 1.
	\end{align*}
According to Fact \ref{fact-1}, this implies $X_n = O_{\PP} (Y_n;~r_n )$.
\end{proof}

Now we relate the new notations to a classical one.

\begin{definition}[The classical $O_{\PP}$ notation  \citep{Van00}]
Let $\{ X_n \}_{n=1}^{\infty}$ and $\{ Y_n \}_{n=1}^{\infty}$ be two sequences of random variables. We write 
\begin{enumerate}
\item $X_n = O_{\PP} (1)$ if
\begin{align*}
\forall \varepsilon>0,~~ \exists M > 0, \text{ s.t. } \PP ( |X_n| \geq M ) \leq \varepsilon, ~~ \forall n \geq 1;
\end{align*}
\item $X_n = O_{\PP} (Y_n)$ if $X_n = Y_n Z_n$ for some $Z_n = O_{\PP} (1)$.
\end{enumerate}
\end{definition}

\begin{fact}
If $\PP (Y_n = 0) = 0$ for all $n$, then $X_n = O_{\PP} (Y_n ;~ 1)$, $X_n = \hat{O}_{\PP} (Y_n ;~ 1)$ and $X_n = O_{\PP} (Y_n)$ are all equivalent.
\end{fact}

The new notations $O_{\PP}(\cdot ;~ \cdot)$ and $\hat{O}_{\PP}(\cdot ;~ \cdot)$ characterize the sizes and tails of random variables. Just like the equivalence between $O_{\PP}(1)$ and tightness \citep{Van00}, $O_{\PP} (1;~ r_n)$ and $\hat{O}_{\PP}(1 ;~ r_n)$ are related to the exponential tightness in large deviation theory \citep{DZe11}.
In the expression $X_n = O_{\PP} (Y_n;~ r_n)$, the relation between $Y_n$ and $r_n$ is determined by properties of $X_n$. Smaller $Y_n$ (in absolute value) comes with smaller $r_n$ and thus larger exceptional probability.
\begin{example}
Let $\{ Z_i \}_{i=1}^{\infty}$ be i.i.d.~$N(0,1)$ random variables and $X_n = \frac{1}{n} \sum_{i=1}^{n} Z_i$. We have $X_n \sim N(0,1/\sqrt{n})$. A Hoeffding-type inequality \citep[Proposition 5.10]{Ver12} asserts the existence of a constant $c>0$ such that
\begin{align*}
\PP ( |X_n| > t \sqrt{r_n / n} ) \leq e \cdot e^{ -c ( t \sqrt{r_n / n} )^2 } = e \cdot e^{ - c r_n t^2 }, \qquad \forall n \geq 0,~ r_n \geq 0,~  t \geq 0.
\end{align*}
Hence $X_n = O_{\PP} ( \sqrt{r_n / n} ;~ r_n )$ for any $r_n>0$. At different resolutions we get different rates of tail decay.
\end{example}

In addition to $O_{\PP}$, we can also extend other classical notations.

\begin{definition}
Let $\{ X_n \}_{n=1}^{\infty}$, $\{ Y_n \}_{n=1}^{\infty}$ be two sequences of random variables, $Y_n \geq 0$ for all $n$, and $\{ r_n \}_{n=1}^{\infty} \subseteq (0,+\infty)$ be deterministic. We write
	\begin{align*}
	X_n = \Omega_{\PP} (Y_n;~ r_n)
	\end{align*}
	if there exist universal constants $(C_1,C_2,N) \in (0,+\infty)^3$ and a non-increasing function $f:~ (0,C_2] \to (0,+\infty)$ satisfying $\lim_{x \to 0} f(0) = +\infty$, such that
	\begin{align*}
	\PP ( X_n \leq t Y_n ) \leq C_1 e^{- r_n f(t) }, \qquad \forall ~n \geq N,~ 0 < t \leq C_2. 
	\end{align*}
\end{definition}

\begin{definition}
	We write
\begin{enumerate}
\item $X_n = o_{\PP} (Y_n;~ r_n)$ if $X_n = O_{\PP} (w_n Y_n ;~ r_n)$ holds for some nonnegative deterministic sequence $\{ w_n \}_{n=1}^{\infty}$ tending to zero;
\item $X_n = \hat{o}_{\PP} (Y_n;~ r_n)$ if $X_n = \hat{O}_{\PP} (w_n Y_n ;~ r_n)$ holds for some nonnegative deterministic sequence $\{ w_n \}_{n=1}^{\infty}$ tending to zero;
\item $X_n = \omega_{\PP} (Y_n;~ r_n)$ if $X_n = \Omega_{\PP} (w_n Y_n ;~ r_n)$ holds for some nonnegative deterministic sequence $\{ w_n \}_{n=1}^{\infty}$ tending to infinity;
\item $X_n = \Theta_{\PP} (Y_n;~ r_n)$ if $X_n = O_{\PP} (Y_n ;~ r_n)$ and $X_n = \Omega_{\PP} (Y_n ;~ r_n)$.
\end{enumerate}
\end{definition}

We conclude this section with some handy results. The proofs of Lemmas \ref{lemma-o}, \ref{lemma-arithmetic} and \ref{lemma-transforms} are straightforward and thus omitted.

\begin{lemma}\label{lemma-o}
	If $X_n = \hat{o}_{\PP} (Y_n;~ r_n)$, then $\lim_{n\to\infty} r_n^{-1} \log \PP ( |X_n| \geq c |Y_n| ) = -\infty$ for any constant $c > 0$. Here we adopt the convention $\log 0 = -\infty$.
\end{lemma}

\begin{lemma}[Addition and multiplication]\label{lemma-arithmetic}
If $X_n = \hat{O}_{\PP} (Y_n ;~ r_n )$ and $W_n = \hat{O}_{\PP} (Z_n ;~ s_n )$, then
\begin{align*}
& X_n + W_n = \hat{O}_{\PP} ( |Y_n| + |Z_n| ;~ r_n \wedge s_n ) , \\
& X_n W_n = \hat{O}_{\PP} ( Y_n Z_n ;~ r_n \wedge s_n ).
\end{align*}
\end{lemma}

\begin{lemma}[Transforms]\label{lemma-transforms}
We have the followings:
\begin{enumerate}
\item if $X_n = \hat{O}_{\PP} (Y_n ;~ r_n)$, then $|X_n|^{\alpha} = \hat{O}_{\PP} ( |Y_n|^{\alpha} ;~ r_n)$ for any $\alpha > 0$;
\item if $X_n = \hat{o}_{\PP} (1 ;~ r_n)$, then $f(X_n) = \hat{o}_{\PP} ( 1 ;~ r_n )$ for any $f:~\RR\to \RR$ that is continuous at $0$.
\end{enumerate}
\end{lemma}

\begin{lemma}[Truncation]
	If $\hat{X}_n = \hat{O}_{\PP} (Y_n ;~ r_n)$ and $\lim_{n\to\infty} r_n^{-1} \log \PP ( |X_n| \geq |\hat{X}_n| ) = -\infty$, then
	\begin{align*}
X_n = \hat{O}_{\PP} (Y_n ;~ r_n).
	\end{align*}
\end{lemma}
\begin{proof}
When $\hat{X}_n = \hat{O}_{\PP} (Y_n ;~ r_n)$, Fact \ref{fact-2} asserts the existence of a constant $C_1>0$ such that
\begin{align*}
\forall C > 0,~~ \exists C' >0 \text{ and } N > 0, \text{ s.t. } \PP ( |\hat{X}_n| \geq C' |Y_n| ) \leq C_1 e^{- C r_n }, \qquad \forall ~n \geq N.
\end{align*}
Fix $C>0$ and find $C'>0$, $N>0$ to make the inequality above holds. Then
\begin{align*}
\PP ( |X_n| \geq C' |Y_n| ) & = \PP ( |X_n| \geq C' |Y_n|,~|X_n| < |\hat{X}_n|  ) + \PP ( |X_n| \geq C' |Y_n|,~|X_n| \geq |\hat{X}_n|  ) \\
&\leq \PP ( |\hat{X}_n| \geq C' |Y_n|  ) + \PP ( |X_n| \geq |\hat{X}_n|  )  \\
& \leq C_1 e^{- C r_n } + \PP ( |X_n| \geq |\hat{X}_n|  ) , \qquad \forall ~n \geq N.
\end{align*}
Since $\lim_{n\to\infty} r_n^{-1} \log \PP ( |X_n| \geq |\hat{X}_n| ) = -\infty$, there exists $N' >0$ such that $\PP ( |X_n| \geq |\hat{X}_n| ) \leq e^{-Cr_n}$ for all $n \geq N' $. As a result,
\begin{align*}
\PP ( |X_n| \geq C' |Y_n| )  \leq C_1 e^{- C r_n } + e^{-Cr_n}
= (C_1 + 1) e^{-Cr_n}, \qquad \forall ~n \geq N \vee N'.
\end{align*}
This proves $X_n = \hat{O}_{\PP} (Y_n ;~ r_n)$.
\end{proof}

\section{Basic examples}\label{section-basic-examples}

In this section, we express common results in the new language. They will serve as building blocks of our advanced examples. For the sake of brevity, we focus on the most important notation $O_{\PP} (\cdot;~\cdot)$.

\begin{example}[From moments to tails]\label{lemma-moments}
If $\EE^{1/r_n} |X_n|^{r_n} \leq Y_n$, then $X_n = O_{\PP} ( Y_n;~ r_n )$.
\end{example}
\begin{proof}
By Markov's inequality,
\begin{align*}
\PP ( |X_n |\geq t Y_n ) = \PP [ |X_n |^{r_n} \geq (t Y_n)^{r_n} ] \leq \frac{ \EE |X_n|^{r_n} }{ (t Y_n )^{r_n} } \leq 1/t^{r_n} = e^{- r_n \log t }, \qquad \forall t > 0.
\end{align*}
\end{proof}

\begin{example}[$\ell_{p}$ norms of random vectors]\label{lemma-ell_p}
Suppose that for any $n \in \ZZ_+$, $r_n > 0$ and $\bX_n \in \RR^n$ is a random vector with $\max_{i \in [n]} \EE^{1/r_n} | X_{ni} |^{r_n} \leq Y_n$. Then
\begin{enumerate}
\item $\| \bX_n \|_{r_n} = O_{\PP} ( n^{1/r_n} Y_n ;~ r_n )$;
\item $\| \bX_n \|_{\infty} = O_{\PP} ( Y_n ;~ r_n )$ as long as $r_n \gtrsim \log n$.
\end{enumerate}
\end{example}
\begin{proof}
By direct calculation,
\begin{align*}
\EE \| \bX_n \|_{r_n}^{r_n} = \sum_{i=1}^{n} \EE |X_{ni}|^{r_n} \leq \sum_{i=1}^{n} Y_n^{r_n} = n Y_n^{r_n}.
\end{align*}
Then $\EE^{1/r_n} \| \bX_n \|_{r_n}^{r_n} \leq n^{1/r_n} Y_n$, from where Example \ref{lemma-moments} leads to $\| \bX_n \|_{r_n} = O_{\PP} ( n^{1/r_n} Y_n ;~ r_n )$. Given the elementary fact below, the tail bounds for $\ell_{p}$ norms readily yields $\ell_{\infty}$ results.
\begin{fact}
	If $n \in \ZZ_+ $, $c>0$ and $p > c \log n$, then $n^{1/p} < e^{1/c}$ and
	\begin{align*}
	\| \bx \|_{\infty} \leq \| \bx \|_{p} \leq e^{1/c} \| \bx \|_{\infty}, \qquad \forall \bx \in \RR^n.
	\end{align*}
\end{fact}
\end{proof}

\begin{example}[Tail bounds via $\| \cdot \|_{\psi_{\alpha}}$]\label{lemma-psi}
Let $\alpha \geq 1$. If $r_n \geq 1$ and $\| X_n \|_{\psi_{\alpha}} \leq 1$ for all $n$, then $X_n = O_{\PP} (  r_n^{1/\alpha} ;~ r_n )$.
\end{example}
\begin{proof}
	The desired result follows from $r_n^{-1/\alpha} \EE^{1/r_n} |X_n|^{r_n} \leq \| X_n \|_{\psi_{\alpha}} \leq 1$ and Example \ref{lemma-moments}.
\end{proof}

\begin{example}[$\ell_{p}$ norms via $\| \cdot \|_{\psi_{\alpha}}$]\label{lemma-ell_p-sub}
Let $\alpha \geq 1$. Suppose that for any $n \in \ZZ_+$, $r_n \geq 1$, $\bX_n \in \RR^n$ is a random vector and $\max_{i \in [n]} \| X_{ni} \|_{\psi_{\alpha}} \leq 1$. Then 
\begin{enumerate}
\item $\| \bX_n \|_{r_n} = O_{\PP} ( n^{1/r_n} r_n^{1/\alpha} ;~ r_n )$;
\item $\| \bX_n \|_{\infty} = O_{\PP} ( r_n^{1/\alpha} ;~ r_n )$ when $r_n \gtrsim \log n$.
\end{enumerate}
\end{example}
\begin{proof}
	The desired result follows from $r_n^{-1/\alpha} \EE^{1/r_n} |X_{ni}|^{r_n} \leq \| X_{ni} \|_{\psi_{\alpha}} \leq 1$ and Example \ref{lemma-ell_p}.
\end{proof}

\begin{example}[$\ell_2$ norms of sub-Gaussian vectors]
Suppose that for any $n \in \ZZ_+$, $\bX_n \in \RR^{d_n}$ is a random vector with $\EE \bX_n = \mathbf{0}$ and $\| \bX_n \|_{\psi_2} \leq 1$. Then $\| \bX_n \|_2 = O_{\PP} ( \sqrt{ r_n } ;~ r_n )$ for any $r_n \gtrsim d_n$.
\end{example}
\begin{proof}
By Theorem 2.1 in \cite{HKZ12}, there exists a constant $c>0$ such that
\begin{align*}
\PP [ \| \bX_n \|_2^2 > c(d_n + 2 \sqrt{d_n t} + 2 t)  ] \leq e^{-t}, \qquad \forall t > 0.
\end{align*}
Then the claim is clearly true.
\end{proof}

\begin{example}[Concentration of the sample mean]\label{lemma-average}
Let $\alpha \in \{1, 2\}$, $\{ X_{ni} \}_{n \in \ZZ_+,~i \in [n] }$ be an array of random variables where for any $n$, $\{ X_{ni} \}_{i=1}^n$ are independent and $\max_{i \in [n]} \| X_{ni} \|_{\psi_{\alpha}} \leq 1$. Define $\bar{X}_n = \frac{1}{n} \sum_{i=1}^{n} X_{ni}$. For any $r_n > 0$, we have
\begin{align*}
\bar X_n - \EE\bar{X}_n = \begin{cases}
O_{\PP} ( \sqrt{ r_n / n } ;~ r_n \wedge n ), &\mbox{ if } \alpha = 1 \\
O_{\PP} ( \sqrt{ r_n / n } ;~ r_n ), &\mbox{ if } \alpha = 2
\end{cases}.
\end{align*}
\end{example}
\begin{proof}
When $\alpha = 1$, a Bernstein-type inequality \citep[Proposition 5.16]{Ver12} asserts the existence of an absolute constant $c>0$ such that
\begin{align*}
\PP ( |\bar X_n - \EE\bar{X}_n| \geq t ) \leq 2 e^{- c n (t^2 \wedge t) }, \qquad \forall n \geq 1,~ t \geq 0.
\end{align*}
When $t = s  \sqrt{ r_n / n }$ with $s \geq 1$, we have $nt = s \sqrt{n r_n} \geq s (r_n \wedge n)$ and $nt^2 = s^2 r_n \geq s (r_n \wedge n)$. Hence 
\begin{align*}
\PP ( |\bar X_n - \EE\bar{X}_n| \geq s  \sqrt{ r_n / n } ) \leq 2 e^{- c (r_n \wedge n) s }, \qquad \forall n \geq 1,~ s \geq 1,
\end{align*}
and $\bar X_n - \EE\bar{X}_n= O_{\PP} (  \sqrt{ r_n / n } ;~ r_n \wedge n )$.

When $\alpha = 2$, a Hoeffding-type inequality \citep[Proposition 5.10]{Ver12} asserts the existence of an absolute constant $c>0$ such that
\begin{align*}
\PP ( |\bar X_n - \EE\bar{X}_n| \geq t ) \leq e \cdot e^{- c t^2 }, \qquad \forall n \geq 1,~ t \geq 0.
\end{align*}
Hence 
\begin{align*}
\PP ( |\bar X_n - \EE\bar{X}_n| \geq s \sqrt{ r_n / n } ) \leq e \cdot e^{- c r_n s^2 }, \qquad \forall n \geq 1,~ s \geq 0,
\end{align*}
and $\bar X_n - \EE\bar{X}_n = O_{\PP} ( \sqrt{ r_n / n } ;~ r_n )$.
\end{proof}

\section{Uniform tail bounds for a collection of random variables}\label{section-uniform-notations}

In this section we present notations and useful results for uniform control over a family of random variables, which is of crucial importance in many applications.

\begin{definition}\label{definition-O-uniform}
	Let $\{ \Lambda_n \}_{n=1}^{\infty}$ be a sequence of finite index sets. For any $n \geq 1$, $\{ X_{n\lambda} \}_{\lambda \in \Lambda_n}$, $\{ Y_{n\lambda} \}_{\lambda \in \Lambda_n}$ are two collections of random variables; $\{ r_{n\lambda} \}_{\lambda \in \Lambda_n} \subseteq (0,+\infty)$ are deterministic. We write
	\begin{align}
	\{ X_{n\lambda} \}_{\lambda \in \Lambda_n} = O_{\PP} ( \{ Y_{n\lambda} \}_{\lambda \in \Lambda_n};~ \{ r_{n\lambda} \}_{\lambda \in \Lambda_n} )
	\label{eqn-definition-O-uniform}
	\end{align}
	if there exist universal constants $(C_1,C_2,N) \in (0,+\infty)^3$ and a non-decreasing function $f:~[C_2,+\infty) \to (0,+\infty)$ satisfying $\lim_{x \to +\infty} f(x) = +\infty$, such that
	\begin{align*}
	\PP ( |X_{n\lambda}| \geq t |Y_{n\lambda}| ) \leq C_1 e^{-r_{n\lambda} f(t) }, \qquad \forall n \geq N,~ \lambda \in \Lambda_n, ~ t \geq C_2.
	\end{align*}
	When $Y_{n\lambda} = Y_n$ and/or $r_{n\lambda} = r_n$ for all $n$ and $\lambda$, we may replace
	$ \{ Y_{n\lambda} \}_{\lambda \in \Lambda_n}$ and/or $\{ r_{n\lambda} \}_{\lambda \in \Lambda_n} $ in (\ref{eqn-definition-O-uniform}) by $Y_n$ and/or $r_n$ for simplicity.
\end{definition}

In a similar manner, we can also define uniform versions of $o_{\PP} (\cdot ;~\cdot)$ and others. When we have a uniform tail bound and the index set is not exceedingly large, the lemma below states that the maximum still satisfies the same tail bound.

\begin{lemma}[Suprema over finite index sets]\label{lemma-O-union}
If $\{ X_{n\lambda} \}_{\lambda \in \Lambda_n} = O_{\PP} ( \{ Y_{n\lambda} \}_{\lambda \in \Lambda_n}  ;~ r_n )$ and $\log |\Lambda_n| \lesssim r_n$, then
\begin{align*}
\max_{\lambda \in \Lambda_n} | X_{n\lambda} | = O_{\PP} \left( \max_{\lambda \in \Lambda_n} |Y_{n\lambda}| ;~r_{n} \right).
\end{align*}
\end{lemma}
\begin{proof}
	By definition, we can find universal constants $(C_1,C_2,N) \in (0,+\infty)^3$ and a non-decreasing function $f:~[C_2,+\infty)$ with $\lim_{x \to +\infty} f(x) = +\infty$, such that
	\begin{align*}
	\PP ( | X_{n \lambda} | \geq t |Y_{n \lambda}| ) \leq C_1 e^{-r_n f(t)}, \qquad \forall n \geq N,~ \lambda \in \Lambda_n, ~ t \geq C_2.
	\end{align*}
	By union bounds,
	\begin{align*}
	\PP \bigg( \max_{\lambda \in \Lambda_n} | X_{n \lambda} | \geq t \max_{\lambda \in \Lambda_n} |Y_{n \lambda}| \bigg)
	& \leq \sum_{\lambda \in \Lambda_n} \PP ( | X_{n \lambda} | \geq t |Y_{n \lambda}| )
	\leq |\Lambda_n| C_1 e^{-r_n f(t)} = C_1 e^{ -r_n [ f(t) - r_n^{-1} \log |\Lambda_n| ] } .
	\end{align*}
	Then the result becomes obvious given $\log |\Lambda_n| \lesssim r_n$.
\end{proof}

\begin{remark}[$\ell_{\infty}$ bounds revisited]
When proving the $\ell_{\infty}$ bounds in Example \ref{lemma-ell_p}, we resorted to the $r_n$-th moments with $r_n \gtrsim \log n$ and then applied $\ell_{p}$ results with $ p = r_n$. Here Lemma \ref{lemma-O-union} leads to a more direct proof.
\end{remark}

Based on Lemma \ref{lemma-O-union}, it is straightforward to use covering arguments to control the suprema of certain stochastic processes with continuous index sets.

\begin{definition}
	Let $( \cS , \rho )$ be a metric space and $\varepsilon > 0$. $\cN \subseteq \cS$ is said to be an $\varepsilon$-net of $\cS$ if for any $\bx \in \cS$ there exists $\by \in \cN$ such that $\rho(\bx , \by ) \leq \varepsilon$.
\end{definition}

\begin{theorem}[Suprema]\label{lemma-O-covering}
	Suppose that for any $n \in \ZZ_+$, $( \cS_n , \rho_n )$ is a metric space and $\cN_n$ is a finite subset of $ \cS_n$; $\{X_n(u) \}_{u \in \cS_n}$ is a collection of random variables; $Y_n$, $M_n$ and $Z_n$ are random variables; $r_n$, $s_n$ and $\varepsilon_n$ are positive and deterministic. If
	\begin{enumerate}
		\item $\{ X_{n}(u) \}_{u \in \cN_n} = O_{\PP} ( Y_{n};~ r_n )$;
		\item $\cN_n$ is an $\varepsilon_n$-net of $\cS_n$ and $\log |\cN_n| \lesssim r_n$;
		\item $M_n = O_{\PP} (Z_n ;~ s_n)$; for any $n \geq 1$ and $(u, v) \in \cS_n \times \cS_n$,
		\begin{align}
		|X_n(u) - X_n(v)| \leq \bigg( \frac{1}{2\varepsilon_n}  \sup_{w \in \cS_n} |X_n(w)| + M_n \bigg) \rho_n(u , v) , \qquad \text{a.s.};
		\label{ineq-lemma-O-covering}
		\end{align}
	\end{enumerate}
	then
	\begin{align*}
\sup_{u \in \cS_n} |X_n (u)| = O_{\PP} ( |Y_n| + \varepsilon_n |Z_n| ; ~ r_n \wedge s_n ).
\end{align*}
\end{theorem}

The tail bound for the supremum in Theorem \ref{lemma-O-covering} is affected by concentration of individual random variables, fineness of the covering, as well as smoothness of the stochastic process. In many applications we may directly find some theoretically tractable random variable $M_n$ dominating the Lipschitz constant of the process, i.e.
\begin{align*}
|X_n(u) - X_n(v)| \leq M_n \rho_n(u , v), \qquad \forall n \in \ZZ_+ ,~ u \in \cS_n ,~ v \in \cS_n.
\end{align*}
In addition, some stochastic processes exhibit certain ``self-bounding'' properties. For instance, if $\bA_n \in \RR^{n \times n}$ is a symmetric random matrix and $X_n(\bu) = \bu^{\top} \bA_n \bu$ for $\bu \in \SSS^{n-1}$, then
\begin{align}
| X_n(\bu) - X_n(\bv) | = |(\bu + \bv)^{\top} \bA_n (\bu - \bv)| \leq \| \bu + \bv \|_2 \| \bA_n \|_2 \| \bu - \bv \|_2 \leq 2 \sup_{\bu \in \SSS^{n - 1} } |X_n(\bu)| \cdot \| \bu - \bv \|_2 .
\label{ineq-lemma-O-covering-1}
\end{align}
When $\varepsilon_n \leq 1/4$, (\ref{ineq-lemma-O-covering}) holds with $M_n = 0$. Similar self-bounding properties have been studied in the literature of concentration inequalities \citep{BLM00}. For the sake of generality, the upper bound in (\ref{ineq-lemma-O-covering}) includes both $M_n$ and the supremum itself. The proof of Theorem \ref{theorem-main} in Section \ref{section-example-uniform} justifies its applicability.

\begin{proof}
In view of Lemma \ref{lemma-O-union}, we have $\max_{u \in \cN_n} |X_n(u)| = O_{\PP} ( Y_n ;~ r_n)$. Define two events for $t \geq 0$:
\begin{align*}
\cA_{nt} = \bigg\{ \max_{u \in \cN_n} |X_n(u)| < t |Y_n| \bigg\}   \qquad \text{and} \qquad \cB_{nt} = \{ |M_n| < t |Z_n| \}.
\end{align*}
We can find universal constants $(C_1,C_2,N) \in (0,+\infty)^3$ and a non-decreasing function $f:~[C_2,+\infty)$ with $\lim_{x \to +\infty} f(x) = +\infty$, such that
	\begin{align*}
	\PP ( \cA_{nt}^{\mathrm{c}} ) \leq C_1 e^{-r_n f(t)} \qquad \text{and} \qquad 
	\PP ( \cB_{nt}^{\mathrm{c}} ) \leq C_1 e^{-s_n f(t)}, \qquad \forall n \geq N, ~ t \geq C_2.
	\end{align*}
	Thanks to the $\varepsilon_n$-net property of $\cN_n$ and the Lipschitz property of $X_n(\cdot)$, on the event $\cA_{nt} \cap \cB_{nt}$ we have
	\begin{align*}
	\sup_{u \in \cS_n} |X_n(u)| \leq \max_{u \in \cN_n} |X_n(u)| + \bigg( \frac{1}{2\varepsilon_n} \sup_{u \in \cS_n} |X_n(u)| + M_n \bigg) \varepsilon_n < t Y_n + \frac{1}{2} \sup_{u \in \cS_n} |X_n(u)| + t \varepsilon_n |Z_n|.
	\end{align*}
	and thus $\sup_{u \in \cS_n} |X_n(u)| < 2t (|Y_n| + \varepsilon_n |Z_n|)$. The proof is then finished by
	\begin{align*}
	\PP \bigg( \sup_{u \in \cS_n} |X_n(u)| \geq 2t(|Y_n| + \varepsilon_n |Z_n|) \bigg)
	\leq \PP (\cA_{nt}^{\mathrm{c}}) + \PP (\cB_{nt}^{\mathrm{c}} ) \leq 2 C_1 e^{-(r_n \wedge s_n) [ f(t) \wedge g(t) ]}, \qquad \forall n \geq N,~ t \geq C_2.
	\end{align*}
\end{proof}

\section{Example: uniform convergence of empirical gradients}\label{section-example-uniform}

In this section we use an example in statistical learning to illustrate how our new expressions and results make derivations quick and clean.

Let $\ell \in C^2 (\RR)$, $| \ell'(0) | \leq 1$ and $\sup_{x \in \RR}|\ell''(x) | \leq  1$. Suppose that for any $n \in \ZZ_+$, $\{ \bX_{ni} \}_{i=1}^n$ are independent random vectors in $\RR^{d_n}$ and $\max_{i \in [n]} \| \bX_{ni} \|_{\psi_2} \leq 1$. Define
\begin{align*}
\hat L_n (\btheta) = \frac{1}{n} \sum_{i=1}^{n} \ell (\btheta^{\top} \bX_{ni})
\end{align*}
and $ L_n (\btheta) = \EE \hat L_n (\btheta)$. We have $\hat{L}_n \in C^2 (\RR^{d_n})$,
\begin{align*}
& \nabla \hat L_n (\btheta) =  \frac{1}{n} \sum_{i=1}^{n}  \bX_{ni}  \ell' ( \btheta^{\top} \bX_{ni} )  , \\
& \nabla^2 \hat L_n (\btheta) =  \frac{1}{n} \sum_{i=1}^{n}  \bX_{ni} \bX_{ni}^{\top} \ell'' ( \btheta^{\top} \bX_{ni} ) .
\end{align*} 

Such loss functions arise in many applications including generalized linear models, projection pursuit, neural networks, etc. We will show that when $R>0$ is a constant and $n \geq C d_n $ for some sufficiently large constant $C$, there exist positive constants $C_1$, $C_2$,  and $N$ such that
\begin{align*}
\PP \bigg( \sup_{ \| \btheta \|_2 \leq R } \| \nabla \hat L_n (\btheta) - \nabla L_n (\btheta) \|_2 \geq C_1 \sqrt{ \frac{ d_n \log (n / d_n) }{n} } \bigg) \leq C_2 \bigg(  \frac{d_n}{n} \bigg)^{d_n}, \qquad \forall n \geq N.
\end{align*}
To begin with, we derive a crude bound on the smoothness of $\nabla \hat{L}_n - \nabla L_n$.

\begin{lemma}\label{lemma-Hessian}
If $d_n \lesssim n$, then
\begin{align*}
\sup_{\btheta \in \RR^{d_n}} \| \nabla^2 \hat L_n (\btheta) \|_2 + \sup_{\btheta \in \RR^{d_n}} \| \nabla^2 L_n (\btheta) \|_2 = O_{\PP} ( 1 ;~ n).
\end{align*}
\end{lemma}
\begin{proof}
Define $X_n(\bu) = \frac{1}{n} \sum_{i=1}^n (\bu^{\top} \bX_{ni})^2$ for $\bu \in \SSS^{d-1}$. It is easily seen from $|\ell''| \leq 1$ that
\begin{align*}
&\| \nabla^2 \hat L_n (\btheta) \|_2 = \sup_{\bu \in \SSS^{d_n - 1} } | \bu^{\top}  \nabla^2 \hat L_n (\btheta) \bu | = \sup_{\bu \in \SSS^{d_n - 1} } \bigg|  \frac{1}{n} \sum_{i=1}^n (\bu^{\top} \bX_{ni})^2 \ell''(\btheta^{\top} \bX_{ni}) \bigg|
\leq \sup_{\bu \in \SSS^{d_n - 1} } X_n(\bu),
\end{align*}
and similarly, $\| \nabla^2 L_n (\btheta) \|_2 \leq \sup_{\bu \in \SSS^{d_n - 1} } \EE X_n(\bu)$.

When $\bu \in \SSS^{d-1}$, $ 2^{-1/2} \EE^{1/2} (\bu^{\top} \bX_{ni})^2 \leq \| \bu^{\top} \bX_{ni} \|_{\psi_2} \leq \| \bX_{ni} \|_{\psi_2} \leq 1$. Hence
\begin{align}
\sup_{\bu \in \SSS^{d_n - 1} } \EE X_n (\bu) 
= \sup_{\bu \in \SSS^{d_n - 1} } \frac{1}{n} \sum_{i=1}^{n} \EE (\bu^{\top} \bX_{ni})^2  \leq 2.
\label{eqn-lemma-Hessian}
\end{align}
Below we use Theorem \ref{lemma-O-covering} to show that $ \sup_{\bu \in \SSS^{d_n - 1} } X_n (\bu) = O_{\PP} (1;~n)$.
\begin{enumerate}
\item Since $\| ( \bu^{\top} \bX_{ni} )^2 \|_{\psi_1} \lesssim \|  \bu^{\top} \bX_{ni} \|_{\psi_2}^2 \leq 1$, Example \ref{lemma-average} forces
\begin{align*}
\{ X_n(\bu) - \EE X_n(\bu) \}_{\bu \in \SSS^{d_n - 1} } = O_{\PP} ( 1 ;~ n ).
\end{align*}
Then (\ref{eqn-lemma-Hessian}) leads to $\{ X_n(\bu) \}_{\bu \in \SSS^{d_n - 1} } = O_{\PP} ( 1 ;~ n )$.
\item Let $\varepsilon_n = 1/4$. According to Lemma 5.2 in \cite{Ver12}, there is an $\varepsilon_n$-net $\cN_n$ of $\SSS^{d_n - 1}$ with cardinality at most $(1 + 2/ \varepsilon_n)^{d_n}$. Then $\log |\cN_n| = d_n \log 9 \lesssim n$.
\item Similar to (\ref{ineq-lemma-O-covering-1}), we have
\begin{align*}
| X_n (\bu) - X_n (\bv) | \leq \frac{1}{2 \varepsilon_n} \sup_{\bw \in \SSS^{d_n - 1} } |X_n (\bw)| \cdot \| \bu - \bv \|_2, \qquad \forall \bu \in \SSS^{d_n - 1} ,~  \bv \in \SSS^{d_n - 1}.
\end{align*}
\end{enumerate}
Based on all these, Theorem \ref{lemma-O-covering} asserts that $\sup_{\bu \in \SSS^{d_n - 1} } X_n (\bu) = O_{\PP} ( 1 ;~ n )$.
\end{proof}

Now we are ready to prove the main result.

\begin{theorem}\label{theorem-main}
For any constant $R > 0$, there exists a constant $C > 0$ such that when $n \geq C d_n $ for all $n$,
\begin{align*}
\sup_{ \| \btheta \|_2 \leq R } \| \nabla \hat L_n (\btheta) - \nabla L_n (\btheta) \|_2 = O_{\PP} \bigg( \sqrt{ \frac{ d_n \log (n / d_n) }{n} } ;~ d_n \log \bigg( \frac{n}{d_n} \bigg) \bigg).
\end{align*}
\end{theorem}

\begin{proof}
Define
\begin{align*}
& X_n (\btheta, \bu) = \langle  \nabla \hat L_n (\btheta) - \nabla L_n (\btheta) , \bu  \rangle , \qquad \forall 
\btheta \in \RR^{d_n},~ \bu \in \SSS^{d_n - 1} , \\
& \cS_n = \{ \bx \in \RR^{d_n}:~ \| \bx \|_2 \leq R \} \times \SSS^{d_n - 1}, \\
& \rho_n ( (\btheta,\bu) , (\bxi, \bv) ) = ( \| \btheta - \bxi \|_2^2 + \| \bu - \bv \|_2^2 )^{1/2}, 
\qquad \forall (\btheta,\bu) \in \cS_n, ~ (\bxi, \bv) \in \cS_n,
\end{align*}
and $r_n = d_n \log (n/d_n)$. Note that
\begin{align*}
\sup_{ \| \btheta \|_2 \leq R } \| \nabla \hat L_n (\btheta) - \nabla L_n (\btheta) \|_2 
= \sup_{ \| \btheta \|_2 \leq R } \sup_{\bu \in \SSS^{d_n - 1} } \langle \nabla \hat L_n (\btheta) - \nabla L_n (\btheta) , \bu \rangle = \sup_{(\btheta, \bu) \in \cS_n} |X_n (\btheta, \bu)|.
\end{align*}
We will invoke Theorem \ref{lemma-O-covering} to bound the supremum.

\begin{enumerate}
\item For any $(\btheta, \bu) \in \cS_n$, we have
\begin{align*}
X_n (\btheta, \bu) = \langle \nabla \hat L_n (\btheta) - \nabla L_n (\btheta) , \bu \rangle =
\frac{1}{n} \sum_{i=1}^{n} (\bu^{\top} \bX_{ni}) \ell'(\btheta^{\top} \bX_{ni}) - \EE [(\bu^{\top} \bX_{ni}) \ell'(\btheta^{\top} \bX_{ni})].
\end{align*}
Since $|\ell'(0)| \leq 1$ and $\ell'$ is 1-Lipschitz, we have $|\ell'(x)| \leq 1 + |x|$, $\forall x$ and
\begin{align*}
& \| (\bu^{\top} \bX_{ni}) \ell'(\btheta^{\top} \bX_{ni}) \|_{\psi_1} 
 \leq \| \bu^{\top} \bX_{ni} \|_{\psi_1} + \|
| \bu^{\top} \bX_{ni}| \cdot |\btheta^{\top} \bX_{ni} |  \|_{\psi_1} \\
& \lesssim \| \bu^{\top} \bX_{ni} \|_{\psi_2} +  \| \bu^{\top} \bX_{ni} \|_{\psi_2} \| \btheta^{\top} \bX_{ni}  \|_{\psi_2} \leq 1 + R, \qquad \forall (\btheta, \bu) \in \cS_n.
\end{align*}
Example \ref{lemma-average} and the assumption $R \lesssim 1$ yield
\begin{align*}
\{ X_n (\btheta, \bu) \}_{ (\btheta, \bu) \in \cS_n } = O_{\PP} ( \sqrt{r_n / n} ;~ r_n \wedge n ).
\end{align*}
\item Let $\varepsilon_n = 2 \sqrt{(R^2 + 1) d_n / n }$. It follows from $\cS_n \subseteq \{ \bx \in \RR^{2d_n}:~ \|x \|_2^2 \leq R^2 + 1 \}$ and Lemma 5.2 in \cite{Ver12} that there is an $\varepsilon_n$-net $\cN_n$ of $\cS_n$ with cardinality at most $(1 + 2 \sqrt{R^2 + 1} / \varepsilon_n)^{2d_n}$. When $n / d_n$ is large enough, we have
\begin{align*}
\log |\cN_n| \leq  2d_n \log (1 + 2 \sqrt{R^2 + 1} / \varepsilon_n) = 2d_n \log (1 + \sqrt{n / d_n}) \lesssim d_n \log ( \sqrt{n / d_n}) = r_n / 2.
\end{align*}
\item Define $\Delta_n (\btheta) = \nabla  \hat L_n (\btheta) - \nabla L_n (\btheta) $. By triangle's inequality,
\begin{align*}
& | X_n(\btheta, \bu) - X_n(\bxi, \bv) | 
 = | \langle \Delta_n(\btheta) , \bu \rangle - \langle \Delta_n(\bxi) , \bv \rangle |
 \leq  |\langle  \Delta_n (\btheta) -  \Delta_n (\bxi) , \bu \rangle| + |\langle \Delta_n (\bxi) , \bu - \bv \rangle| \\
& \leq \| \Delta_n (\btheta) -  \Delta_n (\bxi) \|_2 + \| \Delta_n(\bxi) \|_2 \| \bu - \bv \|_2.
\end{align*}
Let $M_n = \sup_{ \btheta \in \RR^{d_n}} \| \nabla^2 \hat L_n (\btheta) - \nabla^2  L_n (\btheta) \|_2$. Lemma \ref{lemma-Hessian} implies that $M_n = O_{\PP} (1;~n)$. On the one hand,
\begin{align*}
\| \Delta_n (\btheta) -  \Delta_n (\bxi) \|_2 \leq \| [ \nabla \hat L_n(\btheta) - \nabla L_n(\btheta) ] - [ \nabla \hat L_n(\bxi) - \nabla L_n(\bxi) ] \|_2
\leq M_n  \| \btheta - \bxi \|_2.
\end{align*}
On the other hand,
\begin{align*}
\| \Delta_n(\bxi) \|_2 =  \| \nabla  \hat L_n (\bxi) - \nabla L_n (\bxi) \|_2 
= \sup_{\bw \in \SSS^{d_n - 1} } X_n(\bxi, \bw) \leq \sup_{ (\bm{\eta}, \bw) \in \cS_n } |X_n(\bm{\eta}, \bw)| .
\end{align*}
When $ n \geq 16 (R^2 + 1) d_n $, we have $\varepsilon_n \leq 1/2$ and thus
\begin{align*}
| X_n(\btheta, \bu) - X_n(\bxi, \bv) |  \leq \bigg( M_n + \frac{1}{2 \varepsilon_n} \sup_{ (\bm{\eta}, \bw) \in \cS_n } |X_n(\bm{\eta}, \bw)| \bigg) \rho_n ( (\btheta, \bu) , (\bxi, \bv)  ).
\end{align*}
\end{enumerate}
On top of all these, Theorem \ref{lemma-O-covering} implies that
\begin{align*}
\sup_{(\btheta, \bu) \in \cS_n} |X_n (\btheta, \bu)| = O_{\PP} ( \sqrt{r_n / n} + \varepsilon_n ;~( r_n \wedge n ) \wedge n ).
\end{align*}
When $n > d_n$, we have $0 < \log (n / d_n) < n / d_n$ and thus $0 < r_n < n$. Also, $\varepsilon_n \lesssim \sqrt{d_n / n} \lesssim r_n$. Therefore,
\begin{align*}
\sup_{(\btheta, \bu) \in \cS_n} |X_n (\btheta, \bu)| = O_{\PP} ( \sqrt{r_n / n}  ;~ r_n  )
= O_{\PP} ( \sqrt{d_n \log(n / d_n) / n}  ;~ d_n \log(n / d_n) ).
\end{align*}
\end{proof}

\section*{Acknowledgements}
The author thanks Jianqing Fan, Emmanuel Abbe, Miklos Racz and Yuling Yan for helpful discussions and gratefully acknowledges support from the Harold W. Dodds Fellowship.

{
\bibliographystyle{ims}
\bibliography{bib}
}

\end{document}